\documentclass[11pt, reqno]{amsart}

\usepackage{epsfig,amsmath,amsthm,amssymb,graphicx,hyperref}
\usepackage{txfonts}
\usepackage[margin=1.25in]{geometry}
\usepackage[usenames,dvipsnames,svgnames,table]{xcolor}
\numberwithin{equation}{section}
\usepackage[normalem]{ulem}
\usepackage{lineno}

\newtheorem{theorem}{Theorem}[section]
\newtheorem{observation}[theorem]{Observation}
\newtheorem{proposition}[theorem]{Proposition}

\newtheorem{corollary}[theorem]{Corollary}
\newtheorem{lemma}[theorem]{Lemma}
\theoremstyle{definition}
\newtheorem{definition}[theorem]{Definition}
\newtheorem{question}[theorem]{Question}
\newtheorem{remark}[theorem]{Remark}
\newtheorem{example}[theorem]{Example}

\newcommand{\Wh}{\mathrm{Wh}}
\newcommand{\C}{\mathcal{C}}
\newcommand{\T}{\mathcal{T}}

\def\Z{\mathbb{Z}}

\begin{document}
\title[On the Upsilon invariant and satellite knots]{On the Upsilon invariant and satellite knots}

\author{Peter Feller$^\dag$}
\address{Department of Mathematics, ETH Zurich}
\email{peter.feller@math.ch}
\urladdr{https://people.math.ethz.ch/$\sim$pfeller/}

\author{JungHwan Park$^{\dag\dag}$}
\address{School of Mathematics, Georgia Institute of Technology}
\email{junghwan.park@math.gatech.edu }
\urladdr{http://people.math.gatech.edu/$\sim$jpark929/}

\author{Arunima Ray$^{\dag\dag\dag}$}
\address{Max-Planck-Institut f\"{u}r Mathematik}
\email{aruray@mpim-bonn.mpg.de }
\urladdr{http://people.mpim-bonn.mpg.de/aruray}

\date{\today}
\thanks{$^\dag$ Partially supported by Swiss National Science Foundation Grant 155477.}
\thanks{$^{\dag\dag}$ Partially supported by the National Science Foundation grant DMS-1309081.}
\thanks{$^{\dag\dag\dag}$ Partially supported by an AMS--Simons Travel Grant.}

\begin{abstract}
We study the effect of satellite operations on the Upsilon invariant of Ozsv\'{a}th--Stipsicz--Szab\'{o}. We obtain results concerning when a knot and its satellites are independent; for example, we show that the set $\{D_{2^i,1}\}_{i=1}^\infty$ is a basis for an infinite rank summand of the group of smooth concordance classes of topologically slice knots, for $D$ the positive clasped untwisted Whitehead double of any knot with positive $\tau$--invariant, e.g.\ the right-handed trefoil. We also prove that the image of the Mazur satellite operator on the smooth knot concordance group contains an infinite rank subgroup of topologically slice knots.\end{abstract}
\maketitle

\section{Introduction}
Two knots $K$ and $J$ are said to be \emph{smoothly concordant} if they cobound a smoothly embedded annulus in $S^3\times [0,1]$; if they cobound a locally flat annulus, they are said to be \emph{topologically concordant}.
While we will suppress orientations in this text, we note that knots are oriented connected smooth 1-submanifolds (of $S^3$ if not otherwise specified) and in the above definition of concordance the induced orientations of the annulus has to agree on one knot and disagree on the other knot.
A knot is \emph{smoothly slice} if it bounds a smooth disk in $B^4$, or equivalently, is smoothly concordant to the unknot~$U$. Similarly, a knot is \emph{topologically slice} if it bounds a locally flat disk in $B^4$ or is topologically concordant to $U$. Smooth concordance classes of knots form the smooth knot concordance group, denoted $\C$, under the operation of connected sum. Smooth concordance classes of topologically slice knots form a subgroup $\T\subset \C$. The group $\T$ has received substantial attention since, in particular, any non-trivial element of $\T$ may be used to construct an exotic $\mathbb{R}^4$~\cite[Exercise 9.4.23]{GS99}. Our credo is that the distinction between smooth and topological concordance mirrors the distinction between smooth and topological 4--manifolds.

Given a knot $P$ in a solid torus $V$, called a \textit{pattern}, and any knot $K$, called the \textit{companion}, the classical satellite construction yields the \textit{satellite knot} $P(K)$, by tying $V$ into the knot $K$ (see Figure~\ref{fig:satellite}). For example, the \textit{trivial pattern} is the curve core of $V$, and we have $P(K)=K$ for the trivial pattern.
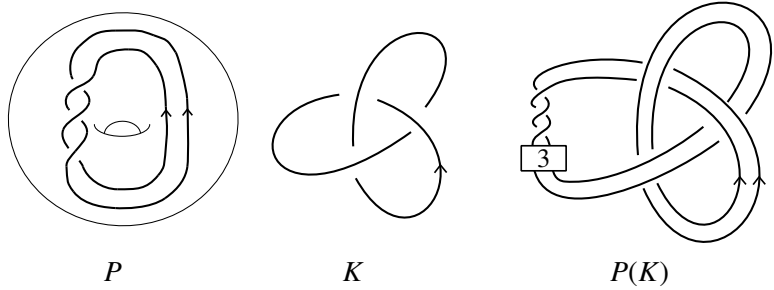
\begin{figure}[h]
  \centering
  \includegraphics[width=4in]{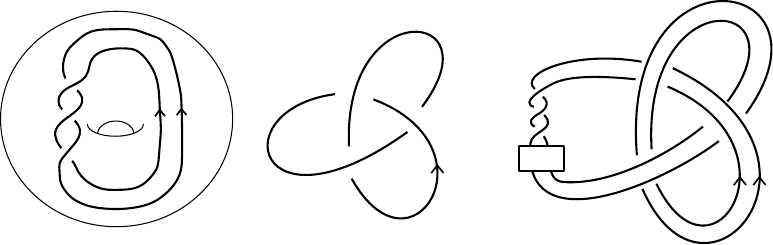}
  \put(-3.5,-0.20){$P$}
	\put(-2.25,-0.20){$K$}
  \put(-0.85,-0.2){$P(K)$}
	\put(-1.23,0.4){\small $3$}
  \caption{The satellite operation on knots. The box containing `$3$' indicates that strands passing vertically through the box are given three full positive twists, to account for the writhe in the given diagram of $K$. Here $P$ has winding number $2$.}\label{fig:satellite}
\end{figure}
Since concordances in either category have tubular neighborhoods, it is easy to see that the satellite operation for any fixed pattern $P$ is well-defined on concordance classes of companions. As a result, any pattern $P$ yields a function on $\mathcal{C}$, called a \textit{satellite operator}, mapping $K\mapsto P(K)$ for any knot $K$. Moreover, whenever $P(U)$ is topologically slice, the pattern $P$ yields a satellite operator $P\colon\T\rightarrow\T$. The \textit{winding number} of a pattern (or satellite operator) $P$ is the algebraic intersection of $P$ with a generic meridional disk of the solid torus containing $P$.

Satellite knots are interesting both within and beyond knot theory. For instance, cable knots were used by Hom~\cite{Hom15} to construct an infinite rank summand of $\T$. Winding number zero satellite operations can be used to construct distinct knot concordance classes which are hard to distinguish using classical invariants~\cite{CHL11,COT04}. They were used in \cite{H08} to modify a 3--manifold without affecting its homology type. Winding number $\pm 1$ satellite operators are related to Mazur 4--manifolds \cite{AkKir79} and Akbulut corks \cite{Ak91}. As a result, there has been considerable interest in understanding the action of satellite operators on $\mathcal{C}$, e.g.\ in the famous conjecture~\cite[Problem 1.38]{kirbylist} that the Whitehead double of a knot $K$ is smoothly slice if and only if $K$ is smoothly slice. Note that all Whitehead doubles are topologically slice~\cite{Freed82}, whereas they are not all smoothly slice~\cite{CG88}.

In~\cite{OzStipSz14}, Ozsv\'{a}th, Stipsicz, and Szab\'{o} introduced the Upsilon invariant of a knot $K$, denoted~$\Upsilon_K$. This invariant considers strictly more of the information contained within the Heegaard Floer complex $CFK^\infty(K)$ of a knot $K$ than the $\tau$--invariant, and is particularly well-suited to studying linear independence of families of knot concordance classes, such as in~\cite{OzStipSz14, Friedl-Livingston-Zentner:15, KimPark16, Wang16, Chen:2016-1}. Our goal in this paper is to demonstrate the power of the Upsilon invariant in addressing questions about linear independence of concordance classes of satellite knots. Our proofs depend on geometric constructions and topological properties of $\Upsilon$ without explicit use of Heegaard Floer theory.

The paper is organized into three independent sections, corresponding to satellite operators with large winding number, winding number one, and winding number zero.
In Section~\ref{sec:otherwindingnumber}, we prove several results about the linear independence of a knot from its cables, and of families of cables and iterated cables. Our main result shows that certain cables of certain knots, e.g.\ the right-handed trefoil, form a basis for an infinite rank summand of $\mathcal{C}$. As a corollary, we show that certain cables of the positive Whitehead double of any knot with positive $\tau$--invariant, e.g.\ the right-handed trefoil, form a basis for an infinite rank summand of $\T$.

\begin{theorem}\label{thm:cablebasis}For any knot $K$ with $\tau(K)=g_c(K)=1$, the knots $\{K_{2^i,1}\}_{i=0}^\infty$ form a basis for an infinite rank summand of $\mathcal{C}$. (Here $g_c(K)$ denotes the concordance genus of $K$.)\end{theorem}

\begin{corollary}\label{cor:whiteheadcablebasis}Let $K$ be a knot with $\tau(K)>0$. Then $\{\Wh^+(K)_{2^i,1}\}_{i=0}^\infty$ is a basis for an infinite rank summand of $\mathcal{T}$. \end{corollary}

Note that the above yields bases for infinite rank summands of $\mathcal{T}$ consisting of knots with Alexander polynomial 1. This should be compared to~\cite{KimPark16} where Kim--Park showed that $\{\Wh^+(RHT)_{n,1}\}_{n=2}^\infty$ is a basis for an infinite rank summand of $\mathcal{C}$. There the authors performed an explicit computation of part of the $\Upsilon$--invariant of $\Wh^+(RHT)_{n,1}$, for $n\geq 2$. Our techniques are comparatively indirect; while we do not address the complete family of $(n,1)$ cables of $\Wh^+(RHT)$, our results apply to a larger family of companion knots. The existence of an infinite rank summand of $\mathcal{T}$ was first shown by Hom in~\cite{Hom15} using her $\varepsilon$--invariant, and later by Ozsv\'{a}th, Stipsicz, and Szab\'{o} in~\cite{OzStipSz14} using their $\Upsilon$--invariant.

In Section~\ref{sec:windingnumberone}, we study winding number one satellite operators. Such operators are of interest since given a winding number one pattern $P$ with $P(U)$ slice, the 0--surgery manifolds of $K$ and $P(K)$ are homology cobordant preserving the homology class of the positively oriented meridian (i.e.\ homology cobordant \textit{rel meridians}); as a result they have the same classical concordance invariants. In~\cite{CFrHeHo13}, it was shown that there exist infinitely many smoothly non-concordant knots with 0--surgeries that are homology cobordant rel meridians. More recently Yasui showed in~\cite{Y15} that there exist smoothly non-concordant knots with homeomorphic 0--surgeries, disproving a conjecture of Akbulut--Kirby from 1978~\cite[Problem 1.19]{kirbylist}. The examples in both~\cite{CFrHeHo13} and~\cite{Y15} are satellite knots. Using the Upsilon invariant, we show the following.

\begin{theorem}\label{cor:cor_yasui}
There is an infinite family of pairs of topologically slice knots $\{(K_i, J_i)\}_{i=0}^\infty$ such that the families $\{K_i\}_{i=0}^\infty$ and $\{J_i\}_{i=0}^\infty$ are each linearly independent in $\mathcal{C}$, each pair $\{K_i, J_i\}$ is linearly independent (with either orientation) in $\mathcal{C}$, and the $0$--surgery manifolds for $K_i$ and $J_i$ are homeomorphic for each $i$.
\end{theorem}

\begin{figure}[t]
\centering
\includegraphics[width=2in]{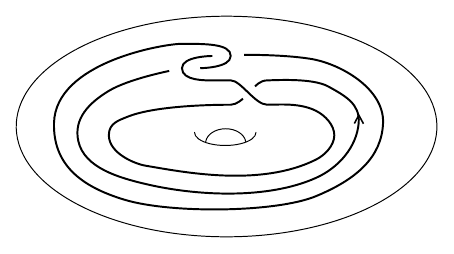}
\caption{The Mazur pattern. We reserve the symbol $M$ for this pattern and the corresponding satellite operator on $\mathcal{C}$ or $\mathcal{T}$, which we call the \textit{Mazur satellite operator}}\label{fig:mazurpattern}
\end{figure}

In~\cite{CHL11}, Cochran--Harvey--Leidy conjectured that some satellite operators are \textit{self-similarities}, which would establish a \textit{fractal} structure on $\mathcal{C}$. This was corroborated in~\cite{CDR14}, where it was shown that infinite classes of winding number one satellite operators (e.g.\ the Mazur satellite operator $M$ shown in Figure~\ref{fig:mazurpattern}) are injective modulo the smooth 4-dimensional Poincar\'{e} Conjecture. Moreover, A.\ Levine showed in~\cite{Lev14} that each iterated function $M^i$ for $i\geq 1$ is non-surjective on $\mathcal{C}$. In contrast, we show that the image of each iterate of the Mazur satellite operator is still `large', even when restricted to $\mathcal{T}$. This is made precise in the following.

\begin{theorem}\label{thm_main}There exists an infinite family of topologically slice knots $\{K_n\}_{n=0}^\infty$ such that, for all non-negative integers $r$, $\{M^r(K_n)\}_{n=0}^\infty$ generates an infinite rank subgroup of $\mathcal{T}$. \end{theorem}

Recall that the Mazur satellite operator is not a group homomorphism (otherwise the result would following from its injectivity). From another perspective, the above result shows that the iterates of the Mazur satellite operator preserve the linear independence of certain families of topologically slice knots.

In Section~\ref{sec:windingnumberszero}, we consider winding number zero satellite operators, particularly the Whitehead doubling operator. As we mentioned earlier, it is a long-standing open question~\cite[Problem 1.38]{kirbylist} whether the Whitehead double of a knot $K$ is smoothly slice if and only if $K$ is smoothly slice. A generalization of this question asks whether Whitehead doubling preserves linear independence. In~\cite{HeK12}, Hedden--Kirk showed that Whitehead doubling preserves the linear independence of certain families of torus knots using moduli spaces of instantons and Chern-Simons invariants of flat connections. Since Upsilon is well-suited to studying linear independence of knots, it is tempting to apply it to the study Whitehead doubles. However, we note in the following proposition that the Upsilon invariant of Whitehead doubles of a knot $K$ contains no more information than the $\tau$--invariant of $K$; in particular, the Upsilon invariant cannot be used to recover Hedden and Kirk's result.

\begin{proposition}\label{prop:whiteheaddoubles} Let $*\in\{+,-\}$, and let $\Wh^*_k(K)$ denotes the $*$--clasped $k$--twisted Whitehead double of the knot $K$. Then,
\[\Upsilon_{\Wh^+_k(K)}(t)=
 \begin{cases}
      0 &\text{if } k\geq 2\tau(K) \\
      -1+\lvert 1-t\rvert &\text{if } k < 2\tau(K)
   \end{cases}
\]
and
\[\Upsilon_{\Wh^-_k(K)}(t)=
 \begin{cases}
      0 &\text{if } k\leq 2\tau(K) \\
      1-\lvert 1-t\rvert &\text{if } k > 2\tau(K).
   \end{cases}
\]
\end{proposition}
Moreover, we explain why $\Upsilon$ can never be used to show the linear independence of an infinite family of winding number zero satellites with a fixed pattern.

For a knot $K$, let $\gamma_4(K)$ denote the minimum first Betti number of a smoothly embedded, connected, compact, possibly non-orientable surface bounded by $K$ in $B^4$. The quantity $\gamma_4$ is called the \textit{smooth 4--dimensional crosscap number} of $K$. In~\cite[Theorem $2$]{Bat14}, Batson showed that $\gamma_4$ of torus knots can be arbitrarily large. More recently in~\cite[Corollary $1.4$]{OzStipSz15} Ozsv\'{a}th--Stipsicz--Szab\'{o} showed that $\gamma_4$ of connected sums of $T_{3,4}$ is arbitrarily large. The following corollary of Proposition~\ref{prop:whiteheaddoubles} adds to these results.

\begin{corollary}\label{cor:crosscapnumber}Let $K$ be a knot with $\tau(K)>0$. Then, for each $k\geq 1$, $\gamma_4(\#_k\Wh^+(K))\geq k$.

Thus, the smooth 4--dimensional crosscap number of topologically slice knots can be arbitrarily large. \end{corollary}

Lastly, we note in passing that this project was motivated by the following question.

\begin{question}\label{q:main}
Does there exists a knot $K$ such that $\{M^n(K)\}_{n=0}^\infty$ is linearly independent in $\mathcal{C}$, or better, is a basis for an infinite rank summand of $\mathcal{C}$? If yes, can this knot be chosen to be topologically slice?
\end{question}

In~\cite{R15} it was shown that there exist knots $K$, including infinitely many topologically slice knots, such that the iterated satellites $M^n(K)$ are all distinct in smooth concordance; this was shown using the $\tau$--invariant. Since $\Upsilon$ can be seen as a generalization of $\tau$, the above is a natural followup question. We do not currently know the answer to Question~\ref{q:main}.

\subsection*{Acknowledgements}The second author would like to thank his advisor Shelly Harvey for her guidance and helpful discussions. We are grateful to the anonymous referee for the detailed and thoughtful suggestions.

\section{Preliminaries}

For ease of reference, we summarize below some of the properties of the $\Upsilon$--invariant given in~\cite{OzStipSz14, OzStipSz15}. These are the only properties of the $\Upsilon$--invariant that we use in this paper.

\begin{proposition}[\cite{OzStipSz14, OzStipSz15}]\label{prop:upsilon_properties} For any knot $K$, the Upsilon invariant, denoted $\Upsilon_K$, is a piecewise linear function $\Upsilon_K: [0,2]\rightarrow \mathbb{R}$ with the following properties.
\begin{enumerate}
\item \cite[Corollary 1.12]{OzStipSz14} $\Upsilon$ is a concordance invariant, i.e\
\begin{align*}
\Upsilon\colon\mathcal{C}\rightarrow &Cont[0,2]\\
K\mapsto &\Upsilon_K
\end{align*}
is a group homomorphism, where $Cont[0,2]$ is the vector space of continuous functions on $[0,2]$. If $rK$ denotes the reverse of the knot $K$, then $\Upsilon_K=\Upsilon_{rK}$.
\item \cite[Proposition 1.5]{OzStipSz14} $\Upsilon_K(0)=0$.
\item \cite[Proposition 1.2]{OzStipSz14} $\Upsilon_K(t)=\Upsilon_K(2-t)$.
\item \cite[Proposition 1.4, Theorem 1.13]{OzStipSz14} Each of the finitely many slopes of $\Upsilon_K$ is an integer. For any such slope $s$, $$\lvert s\rvert \leq  g_c(K).$$
\item \cite[Proposition 1.6]{OzStipSz14} The slope of $\Upsilon_K$ at $t = 0$ is $-\tau(K)$.
\item \cite[Proposition 1.7]{OzStipSz14} For any $t \in [0, 2]$, $t\cdot \Delta \Upsilon'_K(t)$ is an even integer, where
$$\Delta \Upsilon'_K(t_0) = \lim_{t\rightarrow t_0^+} \Upsilon'_K(t) - \lim_{t\rightarrow t_0^-}\Upsilon'_K(t)$$
 for any $t_0\in[0,2]$.
\item \cite[Proposition 1.10]{OzStipSz14} If $K_-$ is the knot obtained by changing a positive crossing of $K_+$ to a negative crossing, then for $0\leq t\leq 1$,
$$\Upsilon_{K_-}(t)-t \leq \Upsilon_{K_+}(t)\leq \Upsilon_{K_-}(t).$$
\item \cite[Theorem 1.11]{OzStipSz14} For $0 \leq t \leq 1$, $$\lvert \Upsilon_K(t)\rvert \leq  t \cdot g_4(K).$$
\item \cite[Theorem 1.2]{OzStipSz15} Let $\upsilon(K)$ denote the value $\Upsilon_K(1)$, and $\gamma_4(K)$ denote the
\textit{smooth 4--dimensional crosscap number} of $K$. Then

$$\Biggl| \upsilon(K) - \frac{\sigma(K)}{2}\Biggr| \leq \gamma_4(K)$$
\end{enumerate}
\end{proposition}

\begin{observation}\label{obs:slopes}Let $m_1$ and $m_2$ be two consecutive slopes in $\Upsilon_K$ for a knot $K$, i.e.\ there is some point $t_0\in(0,1]$ where the slope of $\Upsilon_K$ changes from $m_1$ to $m_2$, with $m_1\neq m_2$. Then $\lvert m_2-m_1\rvert \geq 2$. Moreover,
\begin{enumerate}
\item $t_0\in \left\{ \dfrac{2}{\lvert m_2-m_1\rvert},\dfrac{4}{\lvert m_2-m_1\rvert},\cdots,\dfrac{2 \lfloor \frac{\lvert m_2-m_1\rvert}{2}\rfloor}{\lvert m_2-m_1\rvert}\right\}$
\item if $\lvert m_2-m_1\rvert =2$, then $t_0=1$ and $m_2=-m_1=\pm 1$.
\end{enumerate}
\end{observation}

\begin{proof}By Proposition~\ref{prop:upsilon_properties}(4), $m_1,m_2\in\mathbb{Z}$. By Proposition~\ref{prop:upsilon_properties}(6), $t_0(m_2-m_1)=2k$ for some $k\in\mathbb{Z}$. If $\lvert m_2-m_1\rvert =1$, then $t_0$ must be an even integer, a contradiction. For (1), note that $t_0=\dfrac{2k}{m_2-m_1}$. For (2), we note that by (1)  when $\lvert m_2-m_1\rvert =2$, $t_0$ is an integer, and therefore $t_0=1$. Then $m_2=-m_1$ by symmetry (Proposition~\ref{prop:upsilon_properties}(3)). \end{proof}

The above result immediately yields a bound on where the singularities of $\Upsilon$ may occur, which we use in Section~\ref{sec:otherwindingnumber}.

\begin{observation}\label{obs:concgenusbound}For any knot $K$ with $g_c(K)>0$, the singularities of $\Upsilon_K$, if any, are located in $\left[\frac{1}{g_c(K)},\frac{2g_c(K)-1}{g_c(K)}\right]$. \end{observation}
\begin{proof}From Proposition~\ref{prop:upsilon_properties}(4), we know that the largest slope change possible in $\Upsilon_K$ is $2g_c(K)$. The result then follows from Observation~\ref{obs:slopes}(1).
\end{proof}

Since we are primarily interested in linearly independent subsets of abelian groups, we make the notion precise.

\begin{definition}Let $G$ be an abelian group. A subset $A\subseteq G$ is said to be \textit{linearly independent} if every $n$--element subset of $A$ generates a free abelian subgroup of $G$ of rank $n$, for all $n$. A subgroup of $G$ is said to have \textit{infinite rank} if it contains an infinite linearly independent subset.

A subset $A\subseteq G$ is said to be a \textit{basis} for a free summand of $G$ if $A$ is linearly independent and the subgroup $F$ of $G$ that it generates is a
summand of $G$, i.e.\ there exists a subgroup $H\leq G$ such that $G=F\oplus H$.\end{definition}

The following result from~\cite{OzStipSz14} shows how to use $\Upsilon$ to detect linear independence.

\begin{lemma}[\cite{OzStipSz14}]\label{lem:linindep}
Let $\{K_i\}_{i\in I}$, where $I\subseteq \mathbb{N}$, be a family of knots. Assume that for each $i\in I$, there exists a singularity $t_i$ of $\Upsilon_{K_i}$ such that $t_i$ is not a singularity of
$\Upsilon_{K_j}$ for any $j<i$. Then $\{K_i\}_{i\in I}$ is linearly independent in $\mathcal{C}$.

If, in addition, the above $t_i=\frac{p_i}{q_i}$, where $p_i$ and $q_i$ are coprime, can be chosen such that \[\begin{array}{cc}
\frac{\Delta \Upsilon'_{K_i}(t_i)}{2q_i}\in\{1,-1\}&\text{if $p_i$ is odd}\\
\frac{\Delta \Upsilon'_{K_i}(t_i)}{q_i}\in\{1,-1\}&\text{if $p_i$ is even}\end{array}
,\] then $\{K_i\}_{i\in I}$ is a basis for a free direct summand of $\mathcal{C}$. \end{lemma}

\begin{proof}Consider the group homomorphism
\[\lambda\colon \mathcal{C}\to \bigoplus_I\Z\]
given componentwise by
\[\lambda_i(K)=\begin{cases}
\frac{\Delta \Upsilon'_{K}(t_i)}{2q_i}&\text{if $p_i$ is odd,}\\
\frac{\Delta \Upsilon'_{K}(t_i)}{q_i}&\text{if $p_i$ is even,}
\end{cases}
\] for the given $t_i=\frac{p_i}{q_i}$. Note that this is well-defined since $\frac{\Delta \Upsilon'_{K}(t_i)}{q_i}$ (and $\frac{\Delta \Upsilon'_{K}(t_i)}{2q_i}$ if $p_i$ is odd) is an integer, which follows from $\Delta \Upsilon'_{K}(t_i)$ being an integer (by Proposition~\ref{prop:upsilon_properties}(4)) and $p_i\Delta \Upsilon'_{K}(t_i)=2kq_i$ for some integer $k$ (by Proposition~\ref{prop:upsilon_properties}(6)).
By assumption, for any $i\in I$, $\lambda_i(K_i)\neq 0$ and $\lambda_i(K_j)= 0$ for all $j<i$ in $I$. Therefore,
$\{\lambda(K_i)\}_{i\in I}$ is a linearly independent subset of $\bigoplus_I\Z$. This proves the first half of the lemma.

For the second part, by assumption, $\lambda_i(K_i)\in\{1,-1\}$, and thus, $\lambda\colon \mathcal{C}\to \bigoplus_I\Z$ is surjective. Then,
\[0\to \ker(\lambda)\to\mathcal{C}\to\bigoplus_I\Z\to0\]
is a short exact sequence with a splitting given by the (unique) section $\phi\colon \bigoplus_I\Z\to \mathcal{C}$ that maps the basis $(\lambda(K_i))_{i\in I})$ of $\bigoplus_I \Z$ to $(K_i)_{i\in I}$.
\end{proof}

From the above, it is straightforward to see that if the first singularities of the $\Upsilon$--invariants of a family of knots are distinct, the family is linearly independent in $\mathcal{C}$. Moreover, if the slope change at the first singularities is the smallest allowed by Proposition~\ref{prop:upsilon_properties}(6), the family forms a basis for a free summand of $\mathcal{C}$. If the knots happen to be topologically slice, the same proof shows that the family is a basis for a free summand of $\mathcal{T}$ (compare to~\cite[Lemma 6.4]{OzStipSz14}).

\section{Winding number zero satellites}\label{sec:windingnumberszero}

By Proposition~\ref{prop:upsilon_properties}(4), the number of possibilities for $\Upsilon$ of a knot with low concordance genus are quite small. As a trivial example, note that if the concordance genus of a knot $K$ is zero, then $K$ is slice and thus $\Upsilon_K$ is the zero function. For concordance genus one knots, $\Upsilon$ is determined by the $\tau$--invariant as follows.

\begin{proposition}\label{prop:concgenusone}Let $K$ be a knot with $g_c(K)=1$. Then $\Upsilon_K$ is either the zero function (iff $\tau(K)=0$), $\Upsilon_{T_{2,3}}$ (iff $\tau(K)=1$), or $\Upsilon_{-T_{2,3}}$ (iff $\tau(K)=-1$).\end{proposition}

Recall from~\cite{OzStipSz14} that $\Upsilon_{T_{2,3}}(t)=-1+\lvert 1-t\rvert$ and thus $\Upsilon_{-T_{2,3}}(t)=1-\lvert 1-t\rvert$.

\begin{proof}[Proof of Proposition~\ref{prop:concgenusone}] By Observation~\ref{obs:concgenusbound}, the slope of $\Upsilon$ can only change at $t=1$ and the only possible slopes are $1$, $0$, and $-1$ by Proposition~\ref{prop:upsilon_properties}(4). Thus, at $t=1$, either the slope changes from $-1$ to $1=\tau(K)$, from $1$ to $-1=\tau(K)$, or stays constant at $0=-\tau(K)$ by Proposition~\ref{prop:upsilon_properties}(5). Together with Proposition~\ref{prop:upsilon_properties}(2), this shows that $\Upsilon$  behaves as claimed.
\end{proof}

For a knot $K$ with concordance genus two, it is easy to check by Observation~\ref{obs:slopes}, that if $\tau(K)=0$ there is exactly one possibility for $\Upsilon$ (namely, the zero function), while if $\tau(K)$ is $1$ or $-1$ there are two possibilities each, and when $\tau(K)$ is $2$ or $-2$, there are four possibilities each, i.e.\ there are thirteen possible forms for $\Upsilon$. Using the same principle, it is straightforward to see that for any fixed concordance genus, there is a finite list of possibilities for the $\Upsilon$--invariant, although the size of the list grows fast. 

Note that any winding number zero pattern $P\subseteq S^1\times D^2$ bounds an orientable surface within $S^1\times D^2$; let $g(P)$ denote the least genus of such a surface. Then the concordance genus of satellites with pattern $P$ is bounded above by $g(P)$ and thus there  is a (possibly large) finite list of possibilities for $\Upsilon$. This shows that $\Upsilon$ can never be used to show the linear independence of an \emph{infinite} family of winding number zero satellites with a fixed pattern (cf.\ Section~\ref{sec:windingnumberone} where we show the linear independence of infinite families of winding number one satellites with a fixed pattern).

From Proposition~\ref{prop:concgenusone}, we can immediately compute the Upsilon functions of all Whitehead doubles by applying Hedden's calculation of $\tau$~\cite[Theorem 1.5]{He07}. We recall Proposition~\ref{prop:whiteheaddoubles} and prove it.

\newtheorem*{prop:whiteheaddoubles}{Proposition~\ref{prop:whiteheaddoubles}}
\begin{prop:whiteheaddoubles}Let $*\in\{+,-\}$, and let $\Wh^*_k(K)$ denotes the $*$--clasped $k$--twisted Whitehead double of the knot $K$. Then,
\[\Upsilon_{\Wh^+_k(K)}(t)=
 \begin{cases}
      0 &\text{if } k\geq 2\tau(K) \\
      -1+\lvert 1-t\rvert &\text{if } k < 2\tau(K)
   \end{cases}
\]
and
\[\Upsilon_{\Wh^-_k(K)}(t)=
 \begin{cases}
      0 &\text{if } k\leq 2\tau(K) \\
      1-\lvert 1-t\rvert &\text{if } k > 2\tau(K).
   \end{cases}
\]
\end{prop:whiteheaddoubles}

\begin{proof}The Whitehead doubling pattern $\Wh\subseteq S^1\times D^2$ bounds a genus one surface within $S^1\times D^2$. Thus, $g_c(\Wh^*_k(K))\leq g(\Wh^*_k(K))\leq 1$ for any $k\in\mathbb{Z}$. The result then follows from Proposition~\ref{prop:concgenusone} and~\cite[Theorem 1.5]{He07}, which states that if $k\geq 2\tau(K)$, then $\tau(\Wh^+_k(K))=0$ and if $k< 2\tau(K)$, $\tau(\Wh^+_k(K))=1$. For the second statement we use the fact that $\Wh^-_k(K)=m\Wh^+_{-k}(mK)$ and $\tau(mK)=-\tau(K)$.\end{proof}



The above proposition determines $\Upsilon$ of twist knots, which by definition are $\Wh^\pm_k(K)$ with $K$ taken to be the unknot. Additionally, we see that $\Upsilon$ is no better than $\tau$ at distinguishing between Whitehead doubles. We can also address the generalized Whitehead doubles $D_{J,s}(K,k)$ studied by A.\ Levine in~\cite{Lev12}. Levine showed that $\tau(D_{J,s}(K,k))= 1$ if and only if $s<2\tau(J)$ and $k<2\tau(K)$, $\tau(D_{J,s}(K,k))= -1$ if and only if $s>2\tau(J)$ and $k>2\tau(K)$, and $\tau(D_{J,s}(K,k))=0$ otherwise. Moreover, $D_{J,s}(K,0)$ is topologically slice. It was also shown that $g(D_{J,s}(K,k))=1$ by~\cite[Figure 2b]{Lev12}. By Proposition~\ref{prop:concgenusone} we then have the following result.

\begin{corollary}\label{cor:genwhiteheaddoubles}Let $D_{J,s}(K,k)$ denote the generalized Whitehead doubles from~\cite{Lev12}. Then
\[\Upsilon_{D_{J,s}(K,k)}(t)=
 \begin{cases}
      -1+\lvert 1-t\rvert &\text{if } s<2\tau(J) \text{ and } k<2\tau(K)\\
      1-\lvert 1-t\rvert &\text{if } s>2\tau(J) \text{ and } k>2\tau(K)\\
      0 &\text{otherwise} \\
   \end{cases}
\]
\end{corollary}

The above results can be restated as follows. The $\Upsilon$--invariant for knots $K$ with concordance genus one has at most one singularity (which must occur at $t=1$), i.e.
\begin{equation*}\label{eq:upssimple}
\Upsilon_{K}(t)=\tau(K)\cdot(-1+\lvert 1-t\rvert).
\end{equation*}
There are several other well-studied families of knots for which $\Upsilon$ has at most one singularity
, e.g.\ quasi-alternating knots~\cite[Theorem $1.14$]{OzStipSz14}. 
Clearly, concordance classes of knots for which $\Upsilon$ has at most one singularity form a subgroup of $\mathcal{C}$. We end this section by recalling Corollary~\ref{cor:crosscapnumber} and proving it.

\newtheorem*{cor:crosscapnumber}{Corollary~\ref{cor:crosscapnumber}}
\begin{cor:crosscapnumber}Let $K$ be a knot with $\tau(K)>0$. Then, for each $k\geq 1$, $\gamma_4(\#_k\Wh^+(K))\geq k$.

Thus, the smooth 4--dimensional crosscap number of topologically slice knots can be arbitrarily large.  \end{cor:crosscapnumber}

\begin{proof} Any Whitehead double is topologically slice~\cite{Freed82} and thus has signature zero. By Proposition~\ref{prop:whiteheaddoubles}, we have $\Upsilon_{\Wh^+(K)}=(-1+\lvert 1-t\rvert)$; in particular, $\upsilon(\Wh^+(K))=-1$. Then we directly apply Proposition~\ref{prop:upsilon_properties}(1) and (9), to see that
\[k=\lvert\upsilon(\#_k\Wh^+(K))\rvert \leq \gamma_4(\#_k\Wh^+(K)). \qedhere
\]
\end{proof}


Note that by Proposition~\ref{prop:concgenusone}, the above holds for any knot $J$ with signature 0 (e.g.\ any topologically slice knot) and $\tau(J)=g_c(J)=1$ rather than just Whitehead doubles of knots $K$ with $\tau(K)>0$.

\section{Winding number one satellites}\label{sec:windingnumberone}

\begin{proposition}\label{prop:KandP(K)linindep}Let $K$ be a knot with $\tau(K)>0$ such that the first singularity of $\Upsilon_K$ occurs at $t_0<1$. Let $\alpha$ be the slope of $\Upsilon_K$ at $t_0+\epsilon$.
Let $J$ be a knot that can be changed to the knot K by changing $r>0$ positive crossings to negative crossings and for which $\tau(J)=\tau(K)+r$.
Assume that
\begin{enumerate}
\item $\alpha \neq n\tau(K)$ for any positive integer $n$,
\item $t_0\cdot r$ is not an even integer, and
\item $r$ and $\tau(K)$ are coprime.
\end{enumerate}
Then $\{K, J\}$ is linearly independent. \end{proposition}
Note that conditions $(2)$ and $(3)$ above are satisfied trivially if $r=1$. Condition $(1)$ is satisfied if $\Upsilon(K)$ is convex.
\begin{example}\label{ex:KandP(K)linindep}
Let $P$ be a pattern that can be changed to the trivial pattern by
changing $r>0$ positive crossings to negative crossings, $\tau(P(K))=\tau(K)+r$, and
$K$ satisfies the 3 last assumptions from Proposition~\ref{prop:KandP(K)linindep}. Then $\{K,P(K)\}$ is linearly independent since the sequence of $r$ crossing changes turning $P$ into the trivial pattern can be used to transform $P(K)$ to $K$.
\end{example}

\begin{proof}[Proof of Proposition~\ref{prop:KandP(K)linindep}]
By Lemma~\ref{lem:linindep}, if the first singularity of $\Upsilon_{J}$ does not occur at $t_0$, then $\{K, J\}$ is linearly independent. So, we are left with the case where the first singularity of $\Upsilon_{J}$ occurs at $t_0$.
By Proposition~\ref{prop:upsilon_properties}(7),
$$|\Upsilon_{K}(t)-\Upsilon_{J}(t)|\leq  r\cdot t$$ for any $t\leq 1$. Therefore, we see that the slope of $\Upsilon_{J}$ at $t_0+\epsilon$ must be $\alpha-r+\ell$ for some non-negative $\ell\in \mathbb{Z}$ (recall that the slopes of any $\Upsilon$ function are always integers by Proposition~\ref{prop:upsilon_properties}(4)).

Assume towards a contradiction that there exist non-zero integers $p,q$ for which $pJ$ is concordant to $qK$.
This yields
$$\frac{-\tau(K)}{-\tau(J)}=\frac{\alpha}{\alpha-r+\ell},$$
which we rewrite as $\alpha=\tau(K)\cdot\left(\frac{\ell}{r}-1\right)$ using the fact that $\tau(J)=\tau(K)+r$. Note that $\alpha=0$ if and only if $\ell=r$, since $\tau(K)>0$. We will address the situation of $\ell=r$ momentarily. Suppose that $\alpha\neq 0$. Then, since $\alpha$ and $\tau(K)$ are both integers with $(r,\tau(K))=1$, we must have that $r\vert \ell$, and moreover, by hypothesis, $\left(\frac{\ell}{r}-1\right)\leq 0$ and thus, $\ell\leq r$. This implies that either $\ell=0$ or $\ell=r$.

If $\ell=0$, then $\alpha=-\tau(K)$ which contradicts the fact that $t_0$ is a singularity of $\Upsilon_K$.

If $\ell=r$, we saw that  $\alpha=0$.
By Proposition~\ref{prop:upsilon_properties}(6), we know that $t_0\cdot (\alpha+\tau(K))=t_0\cdot\tau(K)$ and $t_0\cdot(\alpha-r+\ell+\tau(J))=t_0\cdot(\tau(K)+r)$ must both be even integers, which is impossible unless $t_0\cdot r$ is an even integer. \end{proof}

Several families of patterns and knots satisfying the requirements of Proposition~\ref{prop:KandP(K)linindep} are given in~\cite{R15}. In particular, we have the following corollary for the Mazur pattern, denoted by $M$.

\begin{corollary}\label{cor:KandM(K)linindep}Let $K$ be a knot with $\tau(K)>0$, such that the first singularity of $\Upsilon_K$ occurs at $t_0<1$. Let $\alpha$ be the slope of $\Upsilon_K$ at $t_0+\epsilon$. Assume that $\alpha\neq n\tau(K)$ for any positive integer $n$. Then $\{K, M(K)\}$ is linearly independent.\end{corollary}

\begin{proof}By~\cite{Lev14}, we know that $\tau(M(K))=\tau(K)+1$. It is easy to see that the pattern $M$ is changed to the trivial pattern by changing a single positive crossing (at the clasp) to a negative crossing. Thus we can apply Proposition~\ref{prop:KandP(K)linindep} with $r=1$. \end{proof}

If we restrict to knots with convex $\Upsilon$, we can go further, as follows.

\begin{corollary}\label{cor:torusknots_linindep}Let $K$ denote the torus knot $T_{p,q}$, where $3\leq p<q$. Then the pair $\{K, M^i(K)\}$ is linearly independent in $\mathcal{C}$ for any $i\geq 1$ such that $i$ is not divisible by $p$ and $(i,\frac{(p-1)(q-1)}{2})=1$. \end{corollary}

\begin{proof} For any $i\geq 1$, since $M$ can be changed to the trivial pattern by changing a single positive crossing, $M^i(K)$ can be changed to $K$ by changing $i$ positive crossings to negative crossings. By \cite{Lev14}, $\tau(M^i(K))=\tau(K)+i$, since $\tau(K)=\frac{(p-1)(q-1)}{2}>0$\cite{OzSz03}.

The first singularity of $\Upsilon_K$ occurs at $t_0=2/p<1$; see e.g.~\cite{Wang15}.  Let $\alpha$ denote the slope of $\Upsilon_K$ at $t_0+\epsilon$. {Since $\Upsilon_K$ is convex for torus knots by~\cite[Theorem~1.15]{OzStipSz14}}, we must have that $\lvert \alpha \rvert  < \tau(K)$, and therefore, $\alpha \neq n\tau(K)$ for any positive integer $n$.  Note that $\frac{2}{p}\cdot i$ is not an even integer exactly when $i$ is not divisible by $p$. Lastly, by hypothesis, $(i,\tau(K))=1$, since $\tau(K)=\frac{(p-1)(q-1)}{2}$.\end{proof}

From the above, we see that $\{T_{p,q},M(T_{p,q})\}$ is linearly independent for all $p,q$ with $3\leq p<q$. As a further example, $\{T_{3,4},M^i(T_{3,4})\}$ is linearly independent whenever $i$ is not a multiple of 3. We may also apply Proposition~\ref{prop:KandP(K)linindep} to families of topologically slice knots. Let $(\text{Wh}^+(T_{2,3}))_{p,q}$ denote the $(p,q)$--cable of the positive clasped untwisted Whitehead double of the right-handed trefoil $T_{2,3}$. For $n\geq 0$, let
$$K_n\coloneqq(\text{Wh}^+(T_{2,3}))_{n+2,2n+3} \# -T_{n+2,2n+3}.$$
In~\cite[Theorem~1.20]{OzStipSz14}, Ozsv\'ath, Stipsicz, Szab\'o showed that these knots 
generate an infinite rank summand of $\mathcal{T}$.
Here it is easy to see that each $K_n$ is topologically slice, since $\text{Wh}^+(T_{2,3})$ is topologically slice  and therefore, $(\text{Wh}^+(T_{2,3}))_{n+2,2n+3}$ is topologically concordant to $T_{n+2,2n+3}$. Consequently, $M^r(K_n)$ is topologically slice for all $n\geq 0$ and $r\geq 1$.

\begin{corollary}\label{cor:topsliceKandP(K)linindep} For each $n\geq 0$, $\{K_n,M(K_n)\}$ is linearly independent in $\mathcal{T}$. \end{corollary}

We give the proof after a lemma, collecting several facts from~\cite{Lev14,OzStipSz14}.

\begin{lemma}\label{lem:Knprops}For the knots $K_n$, $n\geq 0$ defined above, let $t_n$ denote the first singularity of $\Upsilon_{K_n}$, and $\alpha_n$ the slope of $\Upsilon_{K_n}$ at $t_n+\epsilon$. Then,
\begin{enumerate}
\item $\tau(K_n)=n+2$
\item $\tau(M^r(K_n))=\tau(K_n)+r$ for any $r\geq 0$.
\item $t_n=\frac{2}{2n+3}$
\item $\alpha_n=n+1$
\end{enumerate}
\end{lemma}

\begin{proof} By the properties of $\tau$ given in~\cite{OzSz03}, $$\tau(K_n)=\tau((\text{Wh}^+(T_{2,3}))_{n+2,2n+3}) -\tau(T_{n+2,2n+3})$$
We also know from~\cite{OzSz03} that $\tau(T_{n+2,2n+3})=\frac{(n+2-1)(2n+3-1)}{2}=\frac{(n+1)(2n+2)}{2}=(n+1)^2$. We know from~\cite[Theorem 1.4]{He07} that $\tau(\text{Wh}^+(T_{2,3}))=1$, since $\tau(T_{2,3})>0$. Note that
$$\tau(\text{Wh}^+(T_{2,3}))=g(\text{Wh}^+(T_{2,3}))=1$$
and so by~\cite[Proposition 3.6]{Hom14}, we see that $\varepsilon(\text{Wh}^+(T_{2,3}))=\text{sgn}(\tau(\text{Wh}^+(T_{2,3}))=1$. Then by~\cite[Theorem 1]{Hom14},
$$\tau((\text{Wh}^+(T_{2,3}))_{n+2,2n+3})=(n+2)\cdot\tau(\text{Wh}^+(T_{2,3}))+\frac{(n+2-1)(2n+3-1)}{2}=(n+2) + (n+1)^2$$
As a result, $\tau(K_n)=n+2>0$.

By~\cite[Corollary 1.6]{Lev14}, part (2) follows since $\tau(K_n)>0$. For part (3), we know from \cite[Lemma 8.7]{OzStipSz14} that the first singularity of $\Upsilon_{(\text{Wh}^+(T_{2,3}))_{n+2,2n+3}}$ occurs at $\frac{2}{2n+3}$ and \cite[Lemma 8.12]{OzStipSz14} shows that $\Upsilon_{T_{n+2,2n+3}}$ has no singularities in $(0,\frac{2}{n+2})$. Since $\frac{2}{2n+3}<\frac{2}{n+2}$ for any $n\geq 0$, we see that the first singularity of $\Upsilon_{K_n}$ occurs at $\frac{2}{2n+3}$.

From~\cite[Theorem 1.20]{OzStipSz14}, we know that the slope of $\Upsilon_{(\text{Wh}^+(T_{2,3}))_{n+2,2n+3}}$ at $\frac{2}{2n+3}+\epsilon$ is $-(n^2+n)$, and since the slope of $\Upsilon_{T_{n+2,2n+3}}$ at $\frac{2}{2n+3}+\epsilon$ is $-(n+1)^2$, we see that $\alpha_n=n+1$. \end{proof}

\begin{proof}[Proof of Corollary~\ref{cor:topsliceKandP(K)linindep}]
From Lemma~\ref{lem:Knprops}, using our previous notation, we see that $\tau(K_n)>0$, $t_n<1$, and $\alpha_n$ is not an positive integer multiple of $\tau(K_n)$, for each $n\geq 0$. Thus we can apply Corollary~\ref{cor:KandM(K)linindep}. \end{proof}

We can also use the $\Upsilon$--invariant to detect that winding number one satellite operators preserve linear independence of certain families of knots.

\begin{proposition}\label{prop:preserves_linindep}Let $\{J_n\}_{n=0}^\infty$ be a sequence of knots with arbitrarily small first singularity of $\Upsilon_{J_n}$; that is, the first singularity of $\Upsilon_{J_n}$ occurs at $t_n$ such that $\lim_{n\rightarrow \infty} t_n =0$. Assume that the change of slope at $t_n$ is positive.
Let $P$ be a pattern which can be changed to the trivial pattern by changing $r$ positive crossings to negative crossings and for which $\tau(P(J_n))=\tau(J_n)+r$ for each $n$. Then there is a subsequence of knots $\{J_{n_l}\}_{n_l=0}^\infty$ such that $\{P(J_{n_l})\}_{n_l=0}^\infty$ is linearly independent in $\mathcal{C}$.\end{proposition}

\begin{proof}First observe that, for any $n\geq 0$, $\Upsilon_{P(J_n)}$ has its first singularity in $(0,t_n]$, as follows. Assume towards a contradiction that $\Upsilon_{P(J_n)}$ has no singularity in $\left(0,t_n\right]$.
{This means that there exists a $\overline{t}$ such that for all $t$ with $1>\overline{t}\geq t>t_n$ we have
$$\Upsilon_{P(J_n)}(t)=-\tau(P(J_n))\cdot t=-\tau(J_n)\cdot t -r\cdot t.$$
Therefore, we have
\[\Upsilon_{J_n}(t)-\Upsilon_{P(J_n)}(t)>-t\cdot\tau(J_n)-(-\tau(J_n)\cdot t -r\cdot t)=r\cdot t\]
for $t>t_n$ close to $t_n$, where the fact that the slope change at $t_n$ is positive is used for the first inequality.}

However, note that since $P$ can be changed to the trivial pattern by changing $r$ positive crossings to negative crossings, from Proposition~\ref{prop:upsilon_properties}(8), we know that $|\Upsilon_{J_n}(t)-\Upsilon_{P(J_n)}(t)|\leq r\cdot t$, contradicting the previous statement.

Set $J_{n_0}=J_0$ and choose $J_{n_l}$ inductively:  assume $J_{n_0},J_{n_1}\cdots J_{n_l}$ have been chosen such that the first singularity of $\Upsilon_{P(J_{n_i})}$ and $\Upsilon_{P(J_{n_j})}$ are different for $i\neq j$. We set $J_{n_{l+1}}=J_n$ such that $t_n$ is strictly smaller than all of the first singularities of $\Upsilon_{P(J_{n_i})}$ for $i\leq l$. By Lemma~\ref{lem:linindep}, this yields a sequence of knots $J_{n_l}$ as desired. \end{proof}

There are several infinite families of knots known to be linearly independent in $\mathcal{C}$, e.g.\ for the positive torus knots this is established using Levine--Tristram signatures in~\cite[Theorem~1]{Li79}. Since winding number one patterns $P$ with $P(U)$ slice preserve signatures, we can also use the Levine--Tristram signatures to show that the image of the positive torus knots under the Mazur pattern are linearly independent in $\mathcal{C}$. The advantage of the $\Upsilon$--invariant over the classical signature function is that it can detect linear independence of families of topologically slice knots, for whom the signature function (averaged at the roots of the Alexander polynomial) vanishes. We show an example below.

\begin{corollary}\label{cor:MKnlinindep} For the sequence of linearly independent topologically slice knots $\{K_n\}_{n=0}^\infty$ from Lemma~\ref{lem:Knprops} and any positive integer $r$, there is a subsequence $\{K_{n_l}\}_{n_l=0}^\infty$ such that the set
 $\{M^r(K_{n_l})\}_{n_l=0}^\infty$
  is linearly independent in $\mathcal{T}$. In particular, $\{M^r(K_{n})\}_{n=0}^\infty$ generates an infinite rank subgroup of $\mathcal{T}$ for all non-negative integers $r$.\end{corollary}

\begin{proof}Using the notation from Lemma~\ref{lem:Knprops}, we saw that $\lim_{n\rightarrow \infty} t_n=0$, the change of slope at $t_n$ is positive, and $\tau(M^r(K_n))=\tau(K_n)+r$ for any $n\geq 0$ and $r\geq 1$. Since the iterated pattern $M^r$ can be changed to the trivial pattern by changing $r$ positive crossings to negative crossings, we can apply Proposition~\ref{prop:preserves_linindep}.
\end{proof}

We note that there are other methods for showing the linear independence of topologically slice knots; the first example of a family of linearly independent topologically slice knots was given by Endo in~\cite{En95}.

We have now finished the proof of Theorem~\ref{thm_main}.

\newtheorem*{thm_main}{Theorem~\ref{thm_main}}
\begin{thm_main} There exists an infinite family of topologically slice knots $\{K_n\}_{n=0}^\infty$ such that, for all non-negative integers $r$, $\{M^r(K_n)\}_{n=0}^\infty$ generates an infinite rank subgroup of $\mathcal{T}$.
\end{thm_main}

By combining Theorem~\ref{thm_main} and Corollary~\ref{cor:topsliceKandP(K)linindep}, we can prove Theorem~\ref{cor:cor_yasui}.

\newtheorem*{cor_yasui}{Theorem~\ref{cor:cor_yasui}}
\begin{cor_yasui}
There is an infinite family of pairs of topologically slice knots $\{(L_i, J_i)\}_{i=0}^\infty$ such that the families $\{L_i\}_{i=0}^\infty$ and $\{J_i\}_{i=0}^\infty$ are each linearly independent in $\mathcal{C}$, each pair $\{L_i, J_i\}$ is linearly independent (with either orientation) in $\mathcal{C}$, and the $0$--surgery manifolds for $L_i$ and $J_i$ are homeomorphic for each $i$.
\end{cor_yasui}

\begin{figure}[t]
\centering
\includegraphics[width=2in]{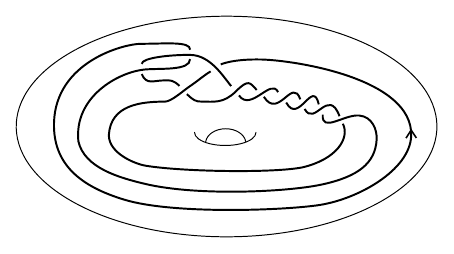}
\caption{The pattern $Q$.}\label{fig:Qpattern}
\end{figure}

\begin{proof}In~\cite[Corollaries 4.12 and 4.13]{Y15}, Yasui showed that for any topologically slice knot $K$ with $\tau(K)>0$, the pair of knots $(M(K),Q(K))$, where $Q$ is the pattern shown in Figure~\ref{fig:Qpattern}, have the same 0--surgery but are not concordant for any orientation. By Corollary~\ref{cor:MKnlinindep}, there is an infinite linearly independent family of topologically slice knots $\{M(K_{n_l})\}_{n_l=0}^\infty$, where $\tau(K_{n_l})>0$. Let $L_i$ denote the knots $M(K_{n_l})$ from Corollary~\ref{cor:MKnlinindep}, and $J_i$ denote the corresponding $Q(K_{n_l})$.

Note that for any knot $K$, $Q(K)$ is concordant to $K$ via cutting a single band. Since each $K_{n_l}$ is topologically slice, so is $Q(K_{n_l})$. Moreover, each pair $\{M(K_{n_l}), Q(K_{n_l})\}$ is linearly independent since $\{M(K_{n_l}), K_{n_l}\}$ is linearly independent by Corollary~\ref{cor:topsliceKandP(K)linindep}. Changing the orientation of the knots does not affect this, since we know from Proposition~\ref{prop:upsilon_properties}(1) that the $\Upsilon$--invariant is invariant under reverses.
\end{proof}

\section{Cable knots}\label{sec:otherwindingnumber}
In this section, $p$ and $q$ are always relatively prime integers with $p>0$. Recall the following formula for signatures of cable knots.

\begin{proposition}[{\cite[Theorem 2]{Li79}}]\label{prop:sigcableformula}Let $K$ be a knot and $\omega$ be a root of unity. Then$$\sigma_\omega(K_{p,q}) = \sigma_{{\omega}^p} (K) + \sigma_\omega(T_{p,q}).$$
\end{proposition}

This immediately yields the following, certainly well-known, result for cables of topologically slice knots.

\begin{corollary}\label{cor:cablesofsliceq>1} Let $p,q$ be coprime integers with $|p|,|q|\geq 2$.
If a topologically slice knot $K$ represents a non-torsion element in $\mathcal{C}$, then $\{K,K_{p,q}\}$ are linearly independent.
\end{corollary}

\begin{proof} Recall that the signature $\sigma=\sigma_{-1}$ is a homomorphism from $\C$ to the integers, which is zero on topologically slice knots and non-zero on non-trivial torus knots $T_{p,q}$. If we have $aK+b K_{p,q}=0$ in $\C$, then applying Proposition~\ref{prop:sigcableformula} for $\omega=-1$ yields $b=0$, since $\sigma_{1}(K)=\sigma_{-1}(K)=0$. This completes the proof since $K$ is assumed to be non-torsion in $\mathcal{C}$.\end{proof}

We are interested in differentiating between topologically slice knots $K$ and its satellites. In light of Corollary~\ref{cor:cablesofsliceq>1}, the interesting case for cables is $q=\pm 1$.

In~\cite{Chen:2016-1}, Chen showed that for any knot $K$, 
\begin{equation}\label{eqn:cheninequality}
\Upsilon_K(pt) - \frac{(p - 1)(q + 1)}{2}t \leq \Upsilon_{K_{p,q}} (t) \leq\Upsilon_K(pt) -\frac{(p - 1)(q - 1)}{2}t,
\end{equation}
when $0\leq t\leq \frac{2}{p}$.

%

For any $K$ with $\tau(K)\neq 0$, Chen's inequality allows us to determine an interval for the first singularity as we see below.

\begin{proposition}\label{prop:cablefirstsingeasy}For any knot $K$ with $\tau(K)\neq 0$, the first singularity of $\Upsilon_{K_{p,q}}$ is in $\left(0,\frac{2}{p}\right]$.\end{proposition}

\begin{proof}
First we consider the case when $\tau(K)>0$. Suppose that $\Upsilon_{K_{p,q}}$ does not have a singularity in $\left(0,\frac{2}{p}\right]$, i.e.\ $\Upsilon_{K_{p,q}}(t)=-\tau(K_{p,q})\cdot t$ on $\left(0,\frac{2}{p}\right]$. Since $\tau(K)>0$, $\varepsilon(K)\neq 0$. If $\varepsilon(K)=1$, we know that $\tau(K_{p,q}) = p\tau(K) + \frac{(p-1)(q-1)}{2}$ from \cite[Theorem 1]{Hom14}, and thus
$$\Upsilon_{K_{p,q}}(t)=-\left(p\tau(K) + \frac{(p-1)(q-1)}{2}\right)\cdot t$$
for any $0\leq t\leq \frac{2}{p}$. By Chen's inequality~(\ref{eqn:cheninequality}), for $0\leq t\leq \frac{2}{p}$,
$$\Upsilon_K(pt)-\frac{(p-1)(q+1)}{2}t\leq -\left(p\tau(K)+\frac{(p-1)(q-1)}{2}\right)\cdot t.$$
By Proposition~\ref{prop:upsilon_properties}(2) and (3), we know that $\Upsilon_K(2)=0$, and thus, at $t=\frac{2}{p}$,
$$\frac{(p-1)(q+1)}{p}\geq 2\tau(K)+\frac{(p-1)(q-1)}{p}$$
and thus,
$$(p-1)\geq p\tau(K)$$
which is a contradiction, since $\tau(K)$ is a positive integer.

If $\varepsilon(K)=-1$, we know that $\tau(K_{p,q}) = p\tau(K) + \frac{(p-1)(q+1)}{2}$ from \cite[Theorem 1]{Hom14}, and thus
$$\Upsilon_{K_{p,q}}(t)=-\left(p\tau(K) + \frac{(p-1)(q+1)}{2}\right)\cdot t$$
for any $0\leq t\leq \frac{2}{p}$. By Chen's inequality~(\ref{eqn:cheninequality}), for $0 \leq t\leq \frac{2}{p}$,
$$\Upsilon_K(pt)-\frac{(p-1)(q+1)}{2}t \leq -\left(p\tau(K)+\frac{(p-1)(q+1)}{2}\right)\cdot t.$$
Thus at $t=\frac{2}{p}$,
$$\frac{(p-1)(q+1)}{p}\geq 2\tau(K)+\frac{(p-1)(q+1)}{p}$$
and thus,
$$0\geq 2\tau(K)$$
which is a contradiction, since $\tau(K)>0$.

For the case when $\tau(K)<0$, notice that $\Upsilon_{K_{p,q}}(t)=-\Upsilon_{-(K_{p,q})}(t)=-\Upsilon_{(-K)_{p,-q}}(t)$. Since $\tau(-K)>0$, the result follows from the previous case.
\end{proof}

For certain knots, we can improve the above result as follows.

\begin{proposition}\label{prop:cablefirstsingugly}For any knot $K$ with $\tau(K)\neq 0$, if the first singularity of $\Upsilon_K$ is located in $\left(1-\frac{p-1}{p(2\lvert\tau(K)\rvert-1)+1},1\right]$, then the first singularity of $\Upsilon_{K_{p,q}}$ is in $\left(0,\frac{2\lvert\tau(K)\rvert}{p(2\lvert\tau(K)\rvert-1)+1}\right]$.\end{proposition}

We get the following immediate corollaries; in the second, we use the fact that $\frac{2\lvert\tau(K)\rvert}{p(2\lvert\tau(K)\rvert-1)+1}<\frac{2}{p}$.

\begin{corollary}If $K$ is a knot with $\tau(K)=\pm1$ for which $\Upsilon$ has exactly one singularity, then the first singularity of $\Upsilon_{K_{p,q}}$ is in $\left(0,\frac{2}{p+1}\right]$. \end{corollary}

\begin{corollary}If $K$ is a knot for which $\Upsilon$ has exactly one singularity, then $\Upsilon_{K_{p,q}}$ has a singularity in $\left(0,\frac{2}{p}\right)$.\end{corollary}

\begin{proof}[Proof of Proposition~\ref{prop:cablefirstsingugly}]The proof is quite similar to the proof for Proposition~\ref{prop:cablefirstsingeasy}. First we consider the case $\tau(K)>0$. By hypothesis, we know that the first singularity of $\Upsilon_K$ occurs in $\left(\frac{2p(\tau(K)-1)+2}{2p\tau(K)-(p-1)},1\right]$, i.e.\ there is no singularity in $\left(0,\frac{2p(\tau(K)-1)+2}{2p\tau(K)-(p-1)}\right]$. By Proposition~\ref{prop:upsilon_properties}(3), there is no singularity in $\left[2-\frac{2p(\tau(K)-1)+2}{2p\tau(K)-(p-1)},2\right)=\left[\frac{2p\tau(K)}{2p\tau(K)-(p-1)},2\right)$. (Note that $\frac{2p\tau(K)}{2p\tau(K)-(p-1)}>1$ since $p\geq 2$). Thus, by Proposition~\ref{prop:upsilon_properties}(3) and (5), for $t\in \left[\frac{2p\tau(K)}{2p\tau(K)-(p-1)},2\right]$,
\begin{equation}\label{eqn:aftersing}
\Upsilon_{K}(t)=\tau(K)(-2+t).
\end{equation}
Suppose that $\Upsilon_{K_{p,q}}$ does not have a singularity in $\left(0,\frac{2\tau(K)}{2p\tau(K)-(p-1)}\right]$, i.e.\ $\Upsilon_{K_{p,q}}(t)=-\tau(K_{p,q})\cdot t$ on $\left(0,\frac{2\tau(K)}{2p\tau(K)-(p-1)}+\epsilon\right)$ for some $\epsilon >0$.
Since $\tau(K)>0$, $\varepsilon(K)\neq 0$. If $\varepsilon(K)=1$, we know that $\tau(K_{p,q}) = p\tau(K) + \frac{(p-1)(q-1)}{2}$ from \cite[Theorem 1]{Hom14}, and thus
$$\Upsilon_{K_{p,q}}(t)=-\left(p\tau(K) + \frac{(p-1)(q-1)}{2}\right)\cdot t$$
for any $0\leq t\leq \frac{2\tau(K)}{2p\tau(K)-(p-1)} + \epsilon$.
By Chen's inequality~(\ref{eqn:cheninequality}), for $\frac{1}{p}\leq t\leq \frac{2\tau(K)}{2p\tau(K)-(p-1)} + \epsilon\leq \frac{2}{p}$,
$$\Upsilon_K(pt)-\frac{(p-1)(q+1)}{2}t \leq -\left(p\tau(K)+\frac{(p-1)(q-1)}{2}\right)\cdot t.$$
For $t=\frac{2\tau(K)}{2p\tau(K)-(p-1)} + \epsilon$, using~(\ref{eqn:aftersing}), we see that
$$\tau(K)(-2+pt)-\frac{(p-1)(q+1)}{2}t\leq -p\tau(K)\cdot t-\frac{(p-1)(q-1)}{2}t$$
and thus,
$$t\leq \frac{2\tau(K)}{2p\tau(K)-(p-1)}$$
which is a contradiction, since $t=\frac{2\tau(K)}{2p\tau(K)-(p-1)} + \epsilon$ and $\epsilon >0$.

If $\varepsilon(K)=-1$, we know that $\tau(K_{p,q}) = p\tau(K) + \frac{(p-1)(q+1)}{2}$ from \cite[Theorem 1]{Hom14}, and thus
$$\Upsilon_{K_{p,q}}(t)=-\left(p\tau(K) + \frac{(p-1)(q+1)}{2}\right)\cdot t$$
for any $0\leq t\leq \frac{2\tau(K)}{2p\tau(K)-(p-1)}+\epsilon$. By Chen's inequality~(\ref{eqn:cheninequality}), for $\frac{1}{p}\leq t\leq \frac{2\tau(K)}{2p\tau(K)-(p-1)}+\epsilon\leq \frac{2}{p}$,
$$\Upsilon_K(pt)-\frac{(p-1)(q+1)}{2}t\leq -\left(p\tau(K)+\frac{(p-1)(q+1)}{2}\right)\cdot t.$$
For $t=\frac{2\tau(K)}{2p\tau(K)-(p-1)} + \epsilon$, using~(\ref{eqn:aftersing}), we see that
$$\tau(K)(-2+pt)-\frac{(p-1)(q+1)}{2}t\leq -p\tau(K)\cdot t-\frac{(p-1)(q+1)}{2}t$$
and thus,
$$t\leq\frac{1}{p}$$
which is a contradiction, since $t=\frac{2\tau(K)}{2p\tau(K)-(p-1)} + \epsilon > \frac{2\tau(K)}{2p\tau(K)-(p-1)} > \frac{1}{p}$, since $\epsilon >0$ and $p\geq 2$.

For the case when $\tau(K)<0$, notice that $\Upsilon_{K_{p,q}}(t)=-\Upsilon_{-(K_{p,q})}(t)=-\Upsilon_{(-K)_{p,-q}}(t)$. Since $\tau(-K)>0$, the result follows from the previous case. \end{proof}

For certain families of knots for which $\Upsilon$ has exactly one singularity, we can further narrow down the location of the first singularity for $(p,1)$ cables, compared to Proposition~\ref{prop:cablefirstsingugly}, as we see below.

\begin{proposition}\label{prop:cablefirstsingbetter}Suppose $K$ is a knot with $g_4(K)=\tau(K)>0$ and for which $\Upsilon$ has exactly one singularity. Then
$\Upsilon_{K_{p,1}}$ has its first singularity in $\left[\frac{1}{p},\frac{2\tau(K)}{2p(\tau(K)-1)+1)}\right]$, for any $p\geq 2$.\end{proposition}

\begin{proof} By 
Proposition~\ref{prop:cablefirstsingugly} it will be enough to show that $\Upsilon_{K_{p,1}}$ does not have a singularity in $\left(0,\frac{1}{p}\right)$. Note that by \cite[Corollary $4$]{Hom14} we have $\varepsilon(K)=1$, and hence $\tau(K_{p,1}) = p\tau(K)$ from \cite[Theorem 1]{Hom14}. For any surface $\Sigma$ in $B^4$ bounded by $K$, we can use $p$ parallel copies band-summed together to get a surface with genus $pg(\Sigma)$ bounded by $K_{p,1}$. Thus, $g_4(K_{p,1})\leq pg_4(K)$. Thus, since $p\tau(K)=\tau(K_{p,1}) \leq g_4(K_{p,1}) \leq pg_4(K)$, we can conclude that $g_4(K_{p,1})= \tau(K_{p,1})=p\tau(K) $.

By Chen's inequality~(\ref{eqn:cheninequality}) for $0\leq t\leq \frac{1}{p}$,
$$\Upsilon_{K_{p,1}}(t) \leq \Upsilon_K(pt)=-p\tau(K)\cdot t.$$
On the other hand by Proposition~\ref{prop:upsilon_properties}(8),
$$-g_4(K_{p,1})\cdot t = -p\tau(K) \cdot t \leq \Upsilon_{K_{p,1}}(t).$$
Thus, for $0\leq t\leq \frac{1}{p}$, $\Upsilon_{K_{p,1}}(t) =-p\tau(K)\cdot t$ which concludes the proof.
\end{proof}


Using the above results, we obtain several corollaries about the linear independence of a knot and its cables. We list several below.

\begin{corollary}\label{cor:cablesindependent}If $K$ is 
a knot for which $\Upsilon$ has exactly one singularity
, then $\{K,K_{p,q}\}$ is linearly independent in $\mathcal{C}$ for any $p,q$ with $p\geq 2$. \end{corollary}

\begin{proof}This follows directly from Lemma~\ref{lem:linindep} and 
Proposition~\ref{prop:cablefirstsingugly}, since $\frac{2\lvert\tau(K)\rvert}{p(2 \lvert\tau(K)\rvert-1)+1}<1$ for any $p\geq 2$ and $\tau(K)\neq 0$. 
\end{proof}

We saw in Section~\ref{sec:windingnumberszero} that there are several families of knots for which $\Upsilon$ has exactly one singularity, including topologically slice knots such as Whitehead doubles~(Proposition~\ref{prop:whiteheaddoubles}) and generalized Whitehead doubles~(Corollary~\ref{cor:genwhiteheaddoubles}). 
For the specific case of Whitehead doubles, we have the following corollary.

\begin{corollary}\label{cor:whiteheadp,pm1cables}Let $\Wh^+(K)$ denote the positive clasped untwisted Whitehead double of a knot $K$ with $\tau(K)>0$. Then the pair $\{\Wh^+(K),\Wh^+(K)_{p,q}\}$ is linearly independent in $\mathcal{T}$ for any $p\geq 2$.\end{corollary}

\begin{remark}Weaker versions of Corollary~\ref{cor:cablesindependent} can be obtained using simple geometric arguments and Proposition~\ref{prop:upsilon_properties}(9). This calculation can be seen in the first version of this paper on the Arxiv, which was written before Chen's inequality was announced.\end{remark}

%

\begin{corollary}\label{cor:Nsmallestbound}Let $K$ be a knot with $\tau(K)\neq 0$ such that the first singularity of $\Upsilon_K$ occurs at $t_0$. Then $K$ and $K_{p,q}$ are linearly independent in $\mathcal{C}$ for any $p> 2/t_0$.\end{corollary}

\begin{proof}By Proposition~\ref{prop:cablefirstsingeasy}, the first singularity of $\Upsilon_{K_{p,q}}$ occurs in $\left(0,\frac{2}{p}\right]$. Since $p>\frac{2}{t_0}$, we see that $\frac{2}{p} < t_0$, and thus we can use Lemma~\ref{lem:linindep}.\end{proof}

\begin{corollary}If $K$ is a knot with $\tau(K)\neq 0$ then for any $p>2g_c(K)$, $K$ and $K_{p,q}$ are linearly independent in $\mathcal{C}$.\end{corollary}

\begin{proof}From Observation~\ref{obs:concgenusbound} we know that the singularities of $\Upsilon_K$ must occur in $\left[\frac{1}{g_c(K)}, \frac{2g_c(K)-1}{g_c(K)}\right]$. By Proposition~\ref{prop:cablefirstsingeasy}, the first singularity of $\Upsilon_{K_{p,q}}$ occurs in $\left(0,\frac{2}{p}\right]$. Since $p>2g_c(K)$, $\frac{2}{p}<\frac{1}{g_c(K)}$, and we can apply Lemma~\ref{lem:linindep}.\end{proof}

\begin{corollary}If $K$ is a knot with $\tau(K)\neq 0$, then there exist a sequence $p_i$ such that for any sequence of integers $q_i$ coprime to $p_i$ the knots $
\{K_{p_i,q_i}\}$ are linearly independent in $\mathcal{C}$. \end{corollary}

\begin{proof}This follows by repeatedly applying Corollary~\ref{cor:Nsmallestbound}. That is, if $t_0$ is the first singularity of $\Upsilon_K$, choose $p_1\geq \frac{2}{t_0}$, then iteratively choose $p_{i+1}\geq \frac{2}{t_i}$ where $t_i$ is the first singularity of $\Upsilon_{K_{p_i,q_i}}$. \end{proof}

\begin{corollary}\label{cor:powercables} Suppose that $K$ has $\tau(K)=g_c(K)>0$ and let $n$ be any integer such that $n>2\tau(K)$. Then $\{K_{n^i,1}\}_{i=0}^\infty$ is linearly independent in $\mathcal{C}$ for $i\geq0$.
\end{corollary}

\begin{proof}Set $p_i\coloneqq n^i$. Using an argument similar to that in the proof of Proposition~\ref{prop:cablefirstsingbetter}, it can be easily verified that $g_c(K_{p_i,1}) \leq p_i\cdot g_c(K)=p_i\cdot \tau(K)$. Further, by using \cite[Theorem 1]{Hom14} $\tau(K_{p_i,1}) =p_i\cdot \tau(K) \leq g_c(K_{p_i,1}) \leq p_i\cdot \tau(K)$ since $\varepsilon(K)=1$ by \cite[Corollary $4$]{Hom14}, and thus, $g_c(K_{p_i,1}) = \tau(K_{p_i,1}) =p_i\cdot \tau(K)$. By Observation~\ref{obs:concgenusbound}, the singularities of $\Upsilon_K$ occur in $\left[\frac{1}{\tau(K)}, \frac{2\tau(K)-1}{\tau(K)}\right]$, and by Observation~\ref{obs:concgenusbound} and Proposition~\ref{prop:cablefirstsingeasy} the first singularity of $\Upsilon_{K_{p_i,1}}$ occurs in $\left[\frac{1}{p_i\cdot\tau(K)}, \frac{2}{p_i}\right]$. Hence, the $\Upsilon_{K_{p_i,1}}$ have distinct singularities since $\frac{2}{p_{i+1}} < \frac{1}{p_i\cdot\tau(K)} $ for all $i\geq0$. We can then apply Lemma~\ref{lem:linindep}.\end{proof}

For 
knots for which $\Upsilon$ has exactly one singularity, we can do better.

\begin{corollary}\label{cor:independentcables}
For any knot $K$ with $g_4(K)=\tau(K)>0$ such that $\Upsilon_K$ has exactly one singularity, the sequence $\{K_{2^i,1}\}_{i=0}^\infty$ is linearly independent in $\C$.
\end{corollary}

\begin{proof}From Proposition~\ref{prop:cablefirstsingbetter}, we know that the first singularity of $\Upsilon_{K_{2^i,1}}$ occurs within $\left[\frac{1}{2^i},\frac{2\tau(K)}{2^{i+1}\tau(K)-(2^{i}-1)}\right]$, and thus are all distinct since $\frac{2\tau(K)}{2^{i+2}\tau(K)-(2^{i+1}-1)}<\frac{1}{2^i}$ for any $i\geq 1$. We can then apply Lemma~\ref{lem:linindep}.\end{proof}

In particular, we get the following.

\begin{corollary}\label{cor:whitehead2^i1cablesindependent} Let $K$ be a knot with $\tau(K)>0$. Then the sequence $\{\Wh^+(K)_{2^i,1}\}_{i=0}^\infty$ is linearly independent in $\mathcal{T}$.\end{corollary}


%

For any knot $K$, let $K_{p_1,1;p_2,1;p_3,1; \dots ; p_n,1}$ denote the iterated cable $(\dots((K_{p_1,1})_{p_2,1})\dots)_{p_n,1}$.

\begin{corollary}\label{cor:iteratedcableslinindep} For any knot $K$ with $\tau(K)\neq 0$, we can find a sequence $\{p_i\}$ such that the family of iterated cables $\{K, K_{p_1,1}, K_{p_1,1;p_2,1}, K_{p_1,1;p_2,1;p_3,1}, ...\}$ is linearly independent in $\mathcal{C}$.\end{corollary}

\begin{proof}Let $t_0$ be the first singularity of $\Upsilon_K$, and $t_i$ the first singularity of $\Upsilon_{K_{p_1,1;p_2,1;...;p_i,1}}$. Choose $p_i$ such that $\frac{2}{p_i} <t_{i-1}$. Now the result follows from Proposition~\ref{prop:cablefirstsingeasy}, where we use the fact that the $\tau$ invariant of the cable knots stays non-zero by \cite[Theorem 1]{Hom14}. \end{proof}



For Whitehead doubles we get the following.

\begin{corollary}\label{cor:whiteheaddoublesiteratedcableslinindep} Let $K$ be a knot with $\tau(K)>0$. Then we can find a sequence $\{p_i\}$ such that the family of iterated cables $\{\Wh^+(K), \Wh^+(K)_{p_1,1}, \Wh^+(K)_{p_1,1;p_2,1}, \Wh^+(K)_{p_1,1;p_2,1;p_3,1}, ...\}$ is linearly independent in $\mathcal{T}$.\end{corollary}

In addition to showing linear independence of families of cables, we can also sometimes conclude that the families form a basis for a free summand of $\mathcal{C}$ or $\mathcal{T}$. We recall and prove Theorems~\ref{thm:cablebasis} and Corollary~\ref{cor:whiteheadcablebasis}.

\newtheorem*{thm:cablebasis}{Theorem~\ref{thm:cablebasis}}
\begin{thm:cablebasis}For any knot $K$ with $\tau(K)=g_c(K)=1$, the knots $\{K_{2^i,1}\}_{i=0}^\infty$ form a basis for an infinite rank summand of $\mathcal{C}$.\end{thm:cablebasis}

\begin{proof}We first note that since $\tau(K)\leq g_4(K)\leq g_c(K)$ for any knot, we have that $\tau(K)=g_4(K)=g_c(K)=1$.

Corollary~\ref{cor:independentcables} proves that the family is linearly independent by showing that the singularities of the $\Upsilon$ functions occur at different locations. Here we will show that the slope change must be the minimum allowed slope change at the first singularity, allowing us to use the second part of Lemma~\ref{lem:linindep}.

Since $\tau(K)=g_c(K)=1$, 
we know (by Proposition~\ref{prop:concgenusone}) that $\Upsilon_K$ has exactly one singularity, which occurs at $1$ and the slope change is $2$, the lowest possible. From Proposition~\ref{prop:cablefirstsingbetter}, we know that the first singularity of $\Upsilon_{K_{2^i,1}}$, for $i\geq0$, occurs in the interval $\left[\frac{1}{2^i},\frac{2}{2^i+1}\right]$. We also know from the proof of Corollary~\ref{cor:powercables} that $g_c(K_{2^i,1})=2^i=\tau(K_{2^i,1})$, and thus the maximum possible slope change in $\Upsilon_{K_{2^i,1}}$ is $2^{i+1}$. Since the first slope equals $-g_c(K_{2^i,1})$, we know that the slope change at the first singularity must be positive. We will now determine the magnitude of the slope change at the first singularity, call it $\Delta m_i$, of $\Upsilon_{K_{2^i,1}}$.

So far we have $0<\Delta m_i \leq 2g_c(K_{2^i,1})=2^{i+1}$.
We know that the singularities occur at points of the form $\frac{2k}{\Delta m_i}$ from Observation~\ref{obs:slopes}. For the first singularity, we have $\frac{1}{2^i}\leq \frac{2k}{\Delta m_i} \leq\frac{2}{2^i+1}$. Thus, $k2^i < k(2^i+1)\leq \Delta m_i \leq 2^{i+1}$ where $k>0$. This implies that $k=1$ and thus $2^i+1\leq \Delta m_i \leq 2^{i+1}$. This means that the first singularity occurs at $\frac{2}{2^i+\ell}$ for some $\ell$ with $1\leq \ell\leq 2^i$, where the slope change must be exactly $2^i+\ell$ (notice that the next higher possible slope change is $2(2^i+\ell)$ but this is not allowed since the slope changes are bounded above by $2^{i+1}=2g_c(K_{2^i,1}))$. Thus the slope change is the lowest possible! We can now use Lemma~\ref{lem:linindep} to conclude that $\{K_{2^i,1}\}_{i=0}^\infty$ is a basis for an infinite summand of $\mathcal{C}$.
\end{proof}

Since $\tau(\Wh^+(K))=g_4(\Wh^+(K))=g_c(\Wh^+(K))=g_3(\Wh^+(K))=1$ for all knots $K$ with $\tau(K)>0$, we have the following.

\newtheorem*{cor:whiteheadcablebasis}{Corollary~\ref{cor:whiteheadcablebasis}}
\begin{cor:whiteheadcablebasis}Let $K$ be a knot with $\tau(K)>0$. Then $\{\Wh^+(K)_{2^i,1}\}_{i=0}^\infty$ is a basis for an infinite rank summand of $\mathcal{T}$. \end{cor:whiteheadcablebasis}


Notice that all the knots in the above corollary have trivial Alexander polynomial. We could also have used generalized Whitehead doubles by Corollary~\ref{cor:genwhiteheaddoubles}. This should be compared to the main result of~\cite{KimPark16}, which showed that the knots $\{\Wh^+(RHT)_{n,1}\}_{n=1}^\infty$ form a basis for an infinite rank summand of $\mathcal{T}$.

So far our results have utilized the non-zero $\tau$--invariant of knots. We now investigate cables of knots with vanishing $\tau$--invariant but non-vanishing $\upsilon$--invariant (recall that for a knot $K$, $\upsilon(K)=\Upsilon_K(1)$).

\begin{proposition}\label{prop:smallupsilon} Let $K$ be a knot with $\tau(K)=0$ and $\upsilon(K)\neq0$. Then the first singularity of $\Upsilon_{K_{p,q}}$ is in $\left(0,\frac{1}{p}\right)$ for any $p\geq 2$.
\end{proposition}

\begin{proof}
First we consider the case $\upsilon(K) <0$. Note that $\upsilon(K)$ is an integer by \cite[Proposition $1.3$]{OzStipSz14}, and thus if $\upsilon(K)<0$ then $\upsilon(K)\leq -1$.
Suppose that $\Upsilon_{K_{p,q}}$ does not have a singularity in $\left(0,\frac{1}{p}\right)$, i.e.\ $\Upsilon_{K_{p,q}}(t)=-\tau(K_{p,q})\cdot t$ on $\left(0,\frac{1}{p}\right]$. If $\varepsilon(K)=1$, we know that $\tau(K_{p,q}) = p\tau(K) + \frac{(p-1)(q-1)}{2}=\frac{(p-1)(q-1)}{2}$ from \cite[Theorem 1]{Hom14}, and thus,
$$\Upsilon_{K_{p,q}}(t)=-\frac{(p-1)(q-1)}{2} t$$
for any $0\leq t\leq \frac{1}{p}$.
By Chen's inequality~(\ref{eqn:cheninequality}), for $0\leq t\leq \frac{1}{p}$,
$$-\frac{(p-1)(q-1)}{2} t \leq \Upsilon_K(pt)-\frac{(p-1)(q-1)}{2}t.$$
When $t=\frac{1}{p}$, we see that $0\leq \upsilon(K)$
which is a contradiction, since $\upsilon(K) \leq -1$.

If $\varepsilon(K)=-1$, we know that $\tau(K_{p,q}) = p\tau(K) + \frac{(p-1)(q+1)}{2}=\frac{(p-1)(q+1)}{2}$ from \cite[Theorem 1]{Hom14}, and thus,
$$\Upsilon_{K_{p,q}}(t)=-\frac{(p-1)(q+1)}{2} t$$
for any $0\leq t\leq \frac{1}{p}$. By Chen's inequality~(\ref{eqn:cheninequality}), for $0 \leq t\leq \frac{1}{p}$,
$$-\frac{(p-1)(q+1)}{2} t \leq \Upsilon_K(pt)-\frac{(p-1)(q-1)}{2}t.$$
When $t=\frac{1}{p}$ we have,
$$-\frac{p-1}{p}\leq \upsilon(K)$$
which is a contradiction, since $p\geq 2$ and $\upsilon(K) \leq -1$.

If $\varepsilon(K)=0$, the proof is the same as the proof for the $\varepsilon(K)=1$ case if $q>0$ and is the same as the proof for the $\varepsilon(K)=-1$ case if $q<0$, by \cite[Theorem 1]{Hom14}.

When $\upsilon(K)>0$, the proof follows from the above by using $-K$, since $\upsilon(-K)=-\upsilon(K) < 0$.
\end{proof}

We have the following immediate corollary.

\begin{corollary}\label{cor:smallupsilon} Let $K$ be a knot with $\tau(K)=0$ and $\upsilon(K)\neq0$ such that the first singularity of $\Upsilon_K$ occurs at $t_0$. Then $K$ and $K_{p,q}$ are linearly independent in $\mathcal{C}$ for any $p>\frac{1}{t_0}$.\end{corollary}

\begin{proof}By Proposition~\ref{prop:smallupsilon}, the first singularity of $\Upsilon_{K_{p,q}}$ occurs in $\left(0,\frac{1}{p}\right)$. Since $p>\frac{1}{t_0}$, we see that $\frac{1}{p} < t_0$, and thus we can use Lemma~\ref{lem:linindep}.\end{proof}

\begin{corollary}\label{cor:smallupsilon4} Let $K$ be a knot with $\tau(K)=0$ and $\upsilon(K)\neq0$. Then there exists a sequence of positive integers $\{p_i\}$ such that for any sequence of integers $q_i$ coprime to $p_i$ the knots $\{K_{p_i,q_i}\}$ are linearly independent in $\mathcal{C}$.\end{corollary}

\begin{proof} By Proposition~\ref{prop:smallupsilon}, the first singularity of $\Upsilon_{K_{p_i,q_i}}$ occurs in $\left(0,\frac{1}{p_i}\right)$ and we can choose $p_i$ such that $\frac{1}{p_i}$ is arbitrarily small. \end{proof}

\begin{remark} In Corollary~\ref{cor:smallupsilon4}, we could have chosen $K$ to be a topologically slice knot. If $\varepsilon(K)=1$ or $0$, choose $q_i=1$ and if $\varepsilon(K)=-1$ choose $q_i=-1$ for all $i$. Then each element in the set $\{K_{p_i,q_i}\}$ has vanishing $\tau$--invariant and each element is topologically slice (in particular, the signature $\sigma$ vanishes). Further any two elements in this set have identical $\varepsilon$ invariant (see \cite[Theorem $2$]{Hom14}), even though they are linearly independent in $\mathcal{T}$ by the above result. \end{remark}

The following lemma produces several topologically slice knots which satisfy the assumptions of Proposition~\ref{prop:smallupsilon}.

\begin{figure}[t]
\centering
\includegraphics[width=2.5in]{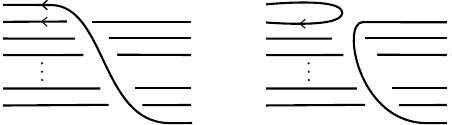}
\caption{For the leftmost full twist in a $(2i,q)$ cabling pattern, perform the non-orientable band sum shown above. Since this non-orientable band sum occurs on the torus the resulting pattern is still a cabling pattern, and further by simply counting the winding number we see that the resulting pattern is a $(2i',q')$ cabling pattern for some $i' <i$.}\label{fig:bandmove}
\end{figure}

\begin{lemma}\label{lem:smallupsilon} Let $K$ be any topologically slice knot such that $g_4(K)=\tau(K)>0$ such that $\Upsilon_K$ has exactly one singularity. Let $J_n = 2nK \# -K_{2n,1}$. Then $J_n$ is a topologically slice knot with $\tau(J_n)=0$ and $\upsilon(J_n)<0$, for all $n \geq 1$.
\end{lemma}

\begin{proof} By \cite[Theorem 1]{Hom14}, $\tau(K_{2n,1})=2n\tau(K)$, and hence $\tau(J_n) = \tau(2nK \# -K_{2n,1}) = 2n\tau(K) - \tau(K_{2n,1})=0$. Note that since $\Upsilon_K$ has exactly one singularity, $\upsilon(K)=-\tau(K)<0$.

By performing the (non-orientable) band sum shown in Figure~\ref{fig:bandmove}, and the fact that any $(2,n)$ cabling pattern bounds a M\"{o}bius band, it is easy to verify that $\gamma_4(K_{2n,1}) \leq n$
(see also~\cite[Section $4$]{Bat14}). Then by Proposition~\ref{prop:upsilon_properties}(10), we have the following inequality.
$$ \Biggl| \upsilon(K_{2n,1}) - \frac{\sigma(K_{2n,1})}{2} \Biggr| \leq \gamma_4(K_{2n,1}) \leq  n
$$
Since $K$ is topologically slice $\sigma(K_{2,1})=\sigma(K)=0$, where we use Proposition~\ref{prop:sigcableformula}. Thus, we see that
$$-\upsilon(K_{2n,1}) \leq n$$
and thus, $$\upsilon(J_n) = 2n\upsilon(K)- \upsilon(K_{2n,1}) =-2n\tau(K)- \upsilon(K_{2n,1})\leq-2n\tau(K) +n.$$

Thus $\upsilon(J_n)<0$ since $\tau(K) >0$ and $n\geq 1$.
\end{proof}

This gives the following corollary, where we only consider the case $q=\pm1$ since it is the most interesting.

\begin{corollary}\label{cor:smallupsilon2} Let $K$ be any topologically slice knot such that $g_4(K)=\tau(K)>0$ such that $\Upsilon_K$ has exactly one singularity. Let $J_n = 2nK \# -K_{2n,1}$, then $\{J_n, (J_n)_{p,\pm1}\}$ is linearly independent in $\mathcal{T}$ for any $p > 2n$.
\end{corollary}

\begin{proof} Note that $\Upsilon_{J_n = 2nK \# -K_{2n,1}}$ has its first singularity in $\left[\frac{1}{2n},\frac{2\tau(K)}{2n\tau(K)-(2n-1)}\right]$ by Proposition~\ref{prop:cablefirstsingbetter}. 
Then the result follows from Corollary~\ref{cor:smallupsilon} and Lemma~\ref{lem:smallupsilon}.
\end{proof}

Note that several knots satisfy the requirements of the above corollary; for instance, we can take $K=\Wh^+(K')$ where $\tau(K')>0$. When $K=\Wh^+(RHT)$ we have the following corollary.

\begin{corollary}\label{cor:smallupsilon3} Let $K=\Wh^+(RHT)$. Let $J_n = 2nK \# -K_{2n,1}$. Then $\{J_n, (J_n)_{p,\pm1}\}$ is linearly independent in $\mathcal{T}$ for any $p > \frac{1+2n}{2}$. In particular $\{J_1, (J_1)_{p,\pm1}\}$ is linearly independent in $\mathcal{T}$ for any $p \geq 2$.
\end{corollary}

\begin{proof}For the most part this follows from Corollary~\ref{cor:smallupsilon2}. We can do a bit better since by~\cite[Theorem C]{KimPark16} and Proposition~\ref{prop:upsilon_properties}(1), $\Upsilon_{J_n}$ has its first singularity at $\frac{2}{1+2n}$. Then we can follow the proof of Corollary~\ref{cor:smallupsilon2}.\end{proof}

\bibliographystyle{alpha}
\bibliography{bib}
\end{document}